\documentclass{article}
\usepackage[english]{babel}
\usepackage{amsmath, amsthm, amssymb}
\usepackage{enumerate}
\usepackage{geometry}
\usepackage{authblk}

\usepackage{hyperref}

\theoremstyle{plain}
\newtheorem{theorem}{Theorem}
\newtheorem{prop}[theorem]{Proposition}
\newtheorem{lemma}[theorem]{Lemma}
\newtheorem{corollary}[theorem]{Corollary}

\theoremstyle{definition}

\newcommand{\C}{\mathbb{C}}
\newcommand{\R}{\mathbb{R}}
\newcommand{\Q}{\mathbb{Q}}
\newcommand{\Z}{\mathbb{Z}}

\let\Re\relax
\DeclareMathOperator{\Re}{Re}
\DeclareMathOperator{\Tr}{Tr}
\DeclareMathOperator{\Nm}{N}

\newcommand{\keywords}[1]{\noindent\textbf{Keywords:} #1}
\newcommand{\subjclass}[2]{\noindent\textbf{#1 Mathematics Subject Classification:} #2.}

\title{Partitions into powers of an algebraic number}
\author[1]{V\'{i}t\v{e}zslav Kala}
\author[2]{Mikul\'{a}\v{s} Zindulka}
\affil[1,2]{Charles University, Faculty of Mathematics and Physics, Department of Algebra, Sokolovsk\'{a} 83, 186 75 Praha 8, Czech Republic}
\affil[1]{Corresponding author. E-mail: vitezslav.kala@matfyz.cuni.cz. ORCID: 000-0001-5515-6801.}
\affil[2]{E-mail: mikulas.zindulka@matfyz.cuni.cz.}

\begin{document}
\maketitle

\begin{abstract}
We study partitions of complex numbers as sums of non-negative powers of a fixed algebraic number $ \beta $. We prove that if $ \beta $ is real quadratic, then the number of partitions is always finite if and only if some conjugate of $ \beta $ is larger than 1. Further, we show that for $ \beta $ satisfying a certain condition, the partition function attains all non-negative integers as values.

\medskip

\medskip
\subjclass{2020}{11P81, 11P84, 11R11}

\end{abstract}

\section{Introduction}

Integer partitions are one of the central objects of study in additive number theory. While the most famous results build on the works of Hardy and Ramanujan and concern the number of partitions without any restrictions, it is also very natural to consider a more general situation.

For a fixed set $ S \subset \C$, one can investigate the properties of partitions of a complex number  into  parts from $ S $. When $ S $ is the set of non-negative powers of a fixed $ m \in \Z_{\geq 2} $, we obtain  the \emph{$ m $-ary partitions}. Thus, an $ m $-ary partition of $ n \in\Z_{\geq 1}$ is an expression of the form
\begin{equation*}
	n = a_jm^j+a_{j-1}m^{j-1}+\dots+a_1m+a_0,
\end{equation*}
where $ j $ and the $ a_i $ are non-negative integers (obviously, other numbers cannot be expressed in this form). The \emph{$ m $-ary partition function}, which counts the number of $ m $-ary partitions of $ n $, is denoted by $ b_m(n) $. Binary partitions (the case $ m = 2 $) were first studied by Euler~\cite{Euler1750}, who defined the binary partition function and computed its values for $ n \leq 37 $ (for more modern results, e.g., see the survey article~\cite{Reznick1990}).

\medskip

The broad goal of this short article is to open up the investigation of general partitions as an analogous, rich area of study. In particular, this will be interesting when $S$ is a subset of some number field $K$, and one considers partitions of elements of $K$.
Before delving into the specifics of our theorems, let us provide more context by discussing various relevant results that one may wish to generalize.

\medskip

The asymptotic behavior of the function $ b_m(n) $ for an integer $ m \geq 2 $ was investigated by Mahler~\cite{Mahler1940}, who proved the asymptotic equality
\begin{equation*}
	\log b_m(n) \sim \frac{(\log n)^2}{2\log m}.
\end{equation*}
(The standard analytic notation $ f \sim g $ means that the ratio of the two functions converges to $ 1 $.) It follows that $ b_m(n) $ grows roughly like $ \exp((\log n)^2) $, and its range is therefore a set of density zero. 

An analogous problem for an arbitrary real $ \beta>1 $ was considered by de Bruijn~\cite{Bruijn1948}, whose work was further improved by Pennington~\cite{Pennington1953}: If one defines $ P_\beta(x) $ to be the number of solutions of the inequality
\begin{equation*}
	a_j\beta^j+a_{j-1}\beta^{j-1}+\cdots+a_1\beta+a_0 < x
\end{equation*}
in non-negative integers, then we similarly have
\begin{equation*}
	\log\left(P_\beta(x)-P_\beta(x-1)\right) \sim \log P_\beta(x) \sim \frac{(\log x)^2}{2\log \beta}.
\end{equation*}

The function $ b_m $ satisfies the folklore recurrences (see for example~\cite{Churchhouse1969})
\begin{align*}
	b_m(nm)& = b_m(nm+1) = \cdots = b_m(nm+(m-1)),\\
	b_m(nm)& = b_m((n-1)m)+b_m(n).
\end{align*}

Another rich area of study are congruence properties of partition functions. In relation to the $ m $-ary partition function, let us just mention a result of Andrews, Fraenkel, and Sellers~\cite{Andrews2015} which characterizes the number of $m$-ary partitions modulo $ m $: If $ n = a_jm^j+\cdots+a_1m+a_0 $ is the base $ m $ representation of $ n $, then
\begin{equation*}
	b_m(mn) \equiv \prod_{i = 0}^j (a_i+1)\pmod{m}.
\end{equation*}
For further interesting results, see~\cite{Churchhouse1969}, \cite{Gupta1972}, \cite{Roedseth1970}, \cite{Rodseth2002}, \cite{Zmija2020}.

\medskip

In the present paper, we focus on partitions into non-negative powers of a fixed complex number $ \beta $. 
For $ \alpha \in \C $, we define $ p_\beta(\alpha) $ as the number of partitions of the form
\begin{equation*}
	\alpha = a_j\beta^j+a_{j-1}\beta^{j-1}+\cdots+a_1\beta+a_0, \text{\ \ \  \ \ \ 	where }  j, a_i \in\Z_{\geq 0}, \text{ and } a_j\neq 0,
\end{equation*} 
equivalently, as the number of polynomials $f(x)\in\Z_{\geq 0}[x]$ such that $f(\beta)=\alpha$. 

If $ \alpha $ cannot be expressed in this form, then $ p_\beta(\alpha) = 0 $. Let us stress that $ p_\beta(\alpha) $ may be equal to infinity (e.g., for $ \beta = 1, -2$, or  $\frac{1}{2} $). Note that this definition is interesting mostly if $ \beta $ is algebraic, for otherwise $ p_\beta(\alpha) \in\{ 0,1\} $ for every $\alpha$. Also when $\beta$ is an $n$th root of unity, then $\beta^n=1$, and so always $ p_\beta(\alpha) \in\{ 0,\infty\} $.
We thus often restrict to the case $ \beta \in \R\setminus\{-1, 0, 1\} $.

\medskip

A related notion is that of a \emph{beta-expansion}, except that one takes the ``digits" in a beta-expansion from a fixed finite alphabet $ \mathcal{A} $ rather than allowing them to be arbitrary non-negative integers (and that one sometimes allows also negative powers of $\beta$). Let $ \beta > 1 $ and let $ \mathcal{A}\subset \R $ be a finite set such that $ 0 \in \mathcal{A} $. The beta-expansion of a non-negative real number $ x $ in the alphabet $ \mathcal{A} $ is
\[
	x = x_j\beta^j+x_{j-1}\beta^{j-1}+\dots+x_1\beta+x_0+x_{-1}\beta^{-1}+x_{-2}\beta^{-2}+\dots,
\]
where $ j \in \Z $ and the $ x_i $ belong to $ \mathcal{A} $. One natural example is taking $ \mathcal{A} = \{0, 1, \dots, \lceil\beta\rceil-1\} $. Beta-expansions were introduced by R\'{e}nyi~\cite{Renyi1957} who considered only greedy expansions. A number usually has more than one expansion in a fixed alphabet, see, e.g.,~\cite{Sidorov2003}. For some further contemporary work on this subject, we refer the reader to~\cite{Frougny1992}, \cite{Pelantova2004b}, \cite{Kalle2012}, \cite{Kovacs1991}, \cite{Schmidt1980}.

Let us finally also mention two kinds of partitions in number fields. One can look at partitions of a totally positive algebraic integer $ \alpha $ as sums of totally positive algebraic integers. An asymptotic formula for the corresponding partition function was found by Meinardus~\cite{Meinardus1953} in the case of a real quadratic field and by Mitsui~\cite{Mitsui1978} in general. Recently, further properties \cite{SZ} and non-trivial identities for these partitions were discovered~\cite{JKK1, JKK2}. Another possibility is to look at partitions as sums of indecomposable elements~\cite{Dress1982}. 
The elements for which there exists exactly one partition of this form are precisely the \emph{uniquely decomposable} elements, which can be characterized in real quadratic fields~\cite{Hejda2020}.

\medskip

Now we return to discussing our results on the function $ p_\beta $. 
The first natural question is whether or not $ p_\beta $ attains only finite values. A sufficient condition for this to be the case is clearly $ \beta>1 $, but this condition is far from necessary. 
In fact, it is easy to observe that for $ \beta $ transcendental, $ p_\beta(\alpha) $ is always 0 or 1. Thus we can restrict ourselves to the case when $ \beta $ is an algebraic number. 
If we want $ p_\beta(\alpha) $ to be finite for every $ \alpha $, it is sufficient for one of its conjugates (i.e., roots of its minimal polynomial) to be greater than $ 1 $ (see Proposition~\ref{propFin}). This is obviously satisfied, e.g., when $ \beta $ is a totally positive algebraic integer (an algebraic number such that its minimal polynomial is monic and all of its conjugates are positive). Our first theorem shows that if $ \beta $ is a real root of a quadratic polynomial, this condition is also necessary.

\begin{theorem}
\label{thm1}
Let $ \beta $ be a real root of an irreducible quadratic polynomial with integer coefficients. Then $ p_\beta(\alpha) < \infty $ for every $ \alpha \in \R $ if and only if at least one of the conjugates $ \beta $ and $ \beta' $ is greater than $ 1 $.
\end{theorem}

What else can we say about the range of the function $ p_\beta $? The answer, under some additional constraints, is provided by our second theorem. Let $ K = \Q(\sqrt{D}) $, where $ D>0 $ is a squarefree integer, be a real quadratic field. We define the \emph{trace} and \emph{norm} of an element $ \beta \in K $ as $ \Tr\beta = \beta+\beta' $ and $ \Nm\beta = \beta\beta' $, where $ \beta' $ denotes the conjugate of $ \beta $. Recall that $ \beta $ is \emph{totally positive} if both $ \beta $ and $ \beta' $ are positive. Finally, if $ \beta $ is the root of a monic irreducible quadratic polynomial over $ \Z $, it is called a \emph{quadratic integer}.

\begin{theorem}
\label{thm2}
Let $ \beta $ be a totally positive quadratic integer such that $ \beta > 1 $ and $ \beta' > 1 $. For every integer $ n \geq 0 $,
\begin{equation}
\label{eqPartTrace}
	p_\beta((\Tr\beta)\beta^n) \geq n+1.
\end{equation}
Equality in~\eqref{eqPartTrace} holds if and only if $ \beta < 2 $ or $ \beta' < 2 $, or $ n \in \{0, 1\} $, or $ n = 2 $ and $ \Nm\beta < 2\Tr\beta $.
\end{theorem}

It is not difficult to show that there exist infinitely many totally positive integral elements $ \beta $ of the field $ K $ satisfying $ \beta > 2 > \beta' > 1 $. Therefore, we obtain the following corollary which is proved in Section \ref{s:3}.

\begin{corollary}
\label{cor3}
In every real quadratic field $ K = \Q(\sqrt{D}) $, there exist infinitely many $ \beta\in \Z[\sqrt{D}] $ such that $ p_\beta $ attains all non-negative integer values.
\end{corollary}

We conclude the paper with several open questions. In particular, 
we do not know of any interesting congruence properties of $ p_\beta $ similar to those for the $ m $-ary partition function in the case when $ \beta $ is not in $ \Z $.

\section*{Acknowledgments}

We thank Jakub Krásenský, Zuzana Mas\'{a}kov\'{a}, and Edita Pelantov\'{a} for helpful discussions, and to Shigeki Akiyama and Artūras Dubickas for suggestions of several useful references. Finally, we thank the anonymous referees for important suggestions and corrections, especially for a simplification of Lemma~\ref{lemmaBetaTotPos} and its proof.

\section{Finiteness of the partition function}\label{s:2}
Let us start by noting that the function $ p_\beta $ satisfies a recurrence relation of sorts.
\begin{prop}
\label{propRec}
If $ \beta \in \C $, then
\[
	p_\beta(\beta\alpha) = p_\beta(\beta\alpha-1)+p_\beta(\alpha)
\]
for every $ \alpha \in \C $.
\end{prop}
\begin{proof}
We give a simple bijective proof. Consider an arbitrary partition
\[
	\beta\alpha = a_j\beta^j+a_{j-1}\beta^{j-1}+\dots+a_1\beta+a_0
\]
of $ \beta\alpha $. If $ a_0 > 0 $, then it corresponds to the partition of $ \beta\alpha-1 $ obtained by subtracting $ 1 $ from $ a_0 $. If $ a_0 = 0 $, then it corresponds to the partition
\[
	\alpha = a_j\beta^{j-1}+a_{j-2}\beta^{j-2}+\dots+a_2\beta+a_1.\qedhere
\]
\end{proof}

We continue by establishing a general sufficient condition on finiteness of values of $p_\beta$.

\begin{prop}
\label{propFin}
Let $ \beta $ be an algebraic number. If one of the conjugates of $ \beta $ is greater than $ 1 $, then $ p_\beta(\alpha) < \infty $ for every $ \alpha \in \R $.
\end{prop}
\begin{proof}
Assume that $ \alpha $ is of the form $ \alpha = a_j\beta^j+a_{j-1}\beta^{j-1}+\dots+a_1\beta+a_0 $, where the $ a_i $ are non-negative integers (if $ \alpha $ is not of this form, then the number of partitions is zero). 

Let $ \beta' $ be the conjugate of $ \beta $ greater than $ 1 $ and consider $ \alpha' = a_j(\beta')^j+a_{j-1}(\beta')^{j-1}+\dots+a_1\beta'+a_0 $. To every partition of $ \alpha $ into powers of $ \beta $ thus corresponds a partition of $ \alpha' $ obtained by replacing $ \beta $ by $ \beta' $. 

But since $ \beta' > 1 $, there are only finitely many partitions of $ \alpha' $ into powers of $ \beta' $. The proposition follows.
\end{proof}

Observe that if $0$ has a non-trivial partition $0=a_j\beta^j+a_{j-1}\beta^{j-1}+\dots+a_1\beta+a_0 $ (i.e., if $\beta$ is the root of some polynomial with non-negative coefficients), then we can repeatedly add this equality to any other partition, showing that $p_\beta(\alpha)\in\{0,\infty\}$ for every $\alpha$.
Such elements $\beta$ were characterized by Dubickas \cite{Du} -- they are precisely those algebraic numbers such that none of their conjugates is a non-negative real number (for more information, see also \cite{Ak, Br, Ha}).

\medskip

In the rest of the paper, $ \beta $ is a root of a quadratic polynomial $ Ax^2+Bx+C $ with integer coefficients. We always assume that $ A > 0 $ and that the polynomial is primitive. By Gauss's lemma, it is irreducible in $ \Q[x] $ if and only if it is irreducible in $ \Z[x] $. If this is the case, we will simply say that it is irreducible.

Everywhere except for Proposition~\ref{propBetaComplex} at the end of this section, $ \beta $ is assumed to be real, which means that the discriminant $ \Delta = B^2-4AC $ is positive. Of course, if $ \beta $ is irrational, then $ \Tr\beta = -B/A $ and $ \Nm\beta = C/A $.

\begin{lemma}
\label{lemmaBetaTotPos}
Let $ \beta $ be a real root of an irreducible polynomial $ Ax^2+Bx+C $, $A,B,C\in\Z$. If
\[
	A > 0,\ \  2A+B > 0, \text{ and } A+B+C>0,
\]
then there is $\alpha\in\R$ such that $ p_\beta(\alpha)=\infty $.
\end{lemma}
\begin{proof}
Let $ i \geq 2 $ be such that $ (A+B+C)i \geq C $. We show that if $ \alpha = -(B\beta+C\beta+C)i $, then $ p_\beta(\alpha) = \infty $. For $ n > i $, we have
\begin{align*}
	\alpha=&-(B\beta+C\beta+C)i = (A\beta^2+B\beta+C)\left(\sum_{k=1}^{i-1}k\beta^{n-k}+\sum_{k=i}^n i\beta^{n-k}\right)-(B\beta+C\beta+C)i\\
	=& \sum_{k=1}^{i-1}\left(Ak\beta^{n-k+2}+Bk\beta^{n-k+1}+Ck\beta^{n-k}\right)+\sum_{k=i}^n\left(Ai\beta^{n-k+2}+Bi\beta^{n-k+1}+Ci\beta^{n-k}\right)-(B\beta+C\beta+C)i\\
	=& A\beta^{n+1}+\sum_{k=0}^{i-3}A(k+2)\beta^{n-k}+\sum_{k=0}^{i-2}B(k+1)\beta^{n-k}+\sum_{k=0}^{i-1}Ck\beta^{n-k}\\
	&+ \sum_{k=i-2}^{n-2} Ai\beta^{n-k}+\sum_{k=i-1}^{n-1} Bi\beta^{n-k}+\sum_{k=i}^n Ci\beta^{n-k}-(B\beta+C\beta+C)i\\
	=& A\beta^{n+1}+\sum_{k=0}^{i-3}(A(k+2)+B(k+1)+Ck)\beta^{n-k}+(Ai+B(i-2+1)+C(i-2))\beta^{n-(i-2)}\\
	&+(Ai+Bi+C(i-1))\beta^{n-(i-1)}+\sum_{k=i}^{n-2}(A+B+C)i\beta^{n-k}\\
	=& A\beta^{n+1}+\sum_{k=0}^{i-2}((A+B+C)k+(2A+B))\beta^{n-k}+((A+B+C)i-C)\beta^{n-(i-1)}+\sum_{k=i}^{n-2}(A+B+C)i\beta^{n-k}.
\end{align*}
The coefficients in the last expression are positive, hence it is a partition of $ \alpha $.
\end{proof}

We are now ready to establish the following result that in particular includes Theorem~\ref{thm1} as the equivalence of i) and iii).
\begin{theorem}
\label{thm5}
Let $ \beta $ be a real root of an irreducible polynomial $ Ax^2+Bx+C $
with $A,B,C\in\Z$, $ A > 0 $. Then the following are equivalent:
\begin{enumerate}[i)]
    \item $ p_\beta(\alpha) < \infty $ for every $ \alpha \in \R $,
    \item $ 2A+B \leq 0 $ or $ A+B+C < 0 $,
    \item at least one of the conjugates $ \beta $ and $ \beta' $ is greater than 1.
\end{enumerate}
\end{theorem}

\begin{proof}
The conditions in ii) are easily seen to be equivalent to
\[
	\frac{-B+\sqrt{B^2-4AC}}{2A} > 1,
\]
which is equivalent to iii). That iii) implies i) follows from Proposition~\ref{propFin}. 

Thus it remains to prove that i) implies ii). Assume that ii) is not satisfied, i.e., $ 2A+B > 0 $ and $ A+B+C \geq 0 $. If $ A+B+C=0 $, then $1$ is a root of the polynomial $ Ax^2+Bx+C $ and the polynomial is not irreducible. Thus we can assume $ A+B+C > 0 $. By Lemma~\ref{lemmaBetaTotPos}, there exists $ \alpha \in \R $ such that $ p_\beta(\alpha) = \infty $.
\end{proof}
Let us now describe a key idea for counting partitions. Suppose that $ \alpha \in \R $ is expressed as
\[
	\alpha = c_j\beta^j+c_{j-1}\beta^{j-1}+\dots+c_1\beta+c_0,
\]
where $ c_i \in \Z_{\geq 0} $. If $ \alpha $ has another partition
\[
	\alpha = b_k\beta^k+b_{k-1}\beta^{k-1}+\dots+b_1\beta+b_0,
\]
then let $ Q $ and $ R $ denote the two polynomials
\begin{align*}
	Q(x)& = b_kx^k+b_{k-1}x^{k-1}+\dots+b_1x+b_0,\\
	R(x)& = c_jx^j+c_{j-1}x^{j-1}+\dots+c_1x+c_0.
\end{align*}
It follows that $ \beta $ is a root of $ Q(x)-R(x) $. If we suppose that $ Ax^2+Bx+C $ is irreducible, then it must divide $ Q(x)-R(x) $.  
Thus there exists a polynomial $ P(x) = a_nx^n+a_{n-1}x^{n-1}+\dots+a_1x+a_0 \in\Z[x]$ such that
\[
	P(x)(Ax^2+Bx+C) = Q(x)-R(x).
\]
The coefficient of $ x^i $ in $ Q(x)-R(x) $ is greater than or equal to $ -c_i $. Conversely, if we find a polynomial $ P(x) \in \Z[x] $ such that the coefficients of $ P(x)(Ax^2+Bx+C) $ satisfy this bound, we can reconstruct a partition of $ \alpha $.

Next, we investigate what happens in the situation when $ p_\beta(\alpha) $ is infinite for some $ \alpha $.
\begin{prop}
Let $ \beta $ be a real root of the irreducible polynomial $ Ax^2+Bx+C=0 $, $ A, B, C \in \Z $, $ A > 0 $.
\begin{enumerate}[i)]
\item If $ B\geq 0 $ and $ C>0 $, then $ p_\beta(\alpha) = \infty $ for all $ \alpha \in \Z[\beta] $.
\item If $ B \geq -C > 0 $, then $ p_\beta(\alpha) \in \{0, 1, \infty\} $ for every $ \alpha \in \R $. Moreover, if $ \alpha $ is of the form
\[
    \alpha = c_j\beta^j+c_{j-1}\beta^{j-1}+\dots+c_1\beta+c_0,
\]
where $ c_i \in \Z_{\geq 0} $, then $ p_\beta(\alpha) = 1 $ if and only if $ c_i \in \{0, \ldots, -C-1\} $ for every $ i $.
\end{enumerate}
\end{prop}

\begin{proof}
i) There is a nontrivial partition $ 0 = A\beta^2+B\beta+C $ of zero. Consequently, every $ \alpha $ which has a partition into non-negative powers of $ \beta $ can be partitioned in infinitely many ways. Since $ -1 = A\beta^2+B\beta+(C-1) $ and $ -\beta = A\beta^3+B\beta^2+(C-1)\beta $, an arbitrary element of $ \Z[\beta] $ has a partition.

ii) For every $ n \geq 3 $, we have
\[
	-C = (A\beta^2+B\beta+C)\sum_{i=0}^{n-2}\beta^i-C = A\beta^n+(A+B)\beta^{n-1}+\sum_{i=2}^{n-2}(A+B+C)\beta^i+(B+C)\beta.
\]
By the assumption $ B \geq -C $, the coefficients in the last expression are non-negative, hence $ -C $ can be partitioned in infinitely many ways. It follows that the same holds for $ -C\beta^n $ with $ n \geq 0 $. We see that if one of the coefficients satisfies $ c_i \geq -C $, then $ p_\beta(\alpha) = \infty $.

Conversely, suppose that $ c_i \in \{0, \ldots, -C-1\} $ for all $ i $. By the remarks preceding the proposition, the partitions of $ \alpha $ correspond to polynomials
\[
	(a_n x^n+a_{n-1}x^{n-1}+\dots+a_1x+a_0)(Ax^2+Bx+C)
\]
such that the coefficient of $ x^i $ is greater than or equal to $ -c_i $ for every $ i $. By assumption, $ -c_i>C $, hence
\[
    C < Ca_i+Ba_{i-1}+Aa_{i-2},
\]
where $ i = 0, 1, \ldots, n+2 $ (with the convention that $ a_i = 0 $ for $ i < 0 $ and $ i > n $).

Without loss of generality, assume $ a_0 \neq 0 $. Then $ C < Ca_0 $ implies $ a_0 < 0 $. We prove by induction that $ a_i < 0 $ for $ i = 0, 1, \dots, n $. Indeed, if $ a_{i-1} < 0 $ and $ a_{i-2} \leq 0 $, then
\[
	C < Ca_i+Ba_{i-1}+Aa_{i-2} \leq Ca_i-Ca_{i-1},
\]
and it follows that $ a_i-a_{i-1} \leq 0 $, hence $ a_i \leq a_{i-1} < 0 $. 

The coefficient of $ x^{n+1} $ equals $ Ba_n+a_{n-1} $, which we estimate as
\[
    Ba_n+a_{n-1} \leq -Ca_n \leq C.
\]
This is a contradiction with the fact that all the coefficients are greater than $ C $. This shows that all the $ a_i $ must be zero and $ \alpha $ has a unique partition.
\end{proof}
 
We can ask whether partitions into powers make sense also for non-real quadratic $ \beta $. The next proposition shows that the answer is negative because the partition function $ p_\beta $ is in this case trivial.
\begin{prop}
\label{propBetaComplex}
Let $ \beta\in\C\setminus\R $ be a complex root of a quadratic polynomial with $\Z$-coefficients. Then $0$ has a nontrivial partition, so that $ p_\beta(0) = \infty $. Thus $ p_\beta(\alpha) \in \{0, \infty\} $ for every $ \alpha \in \C $.
\end{prop}
\begin{proof} 
By considering the argument of $\beta$ as a complex number, we see that there exists $ n \in \Z_{\geq 1} $ such that
\[
  \Tr(\beta^n)=  \beta^n+(\beta')^n = 2\Re{\beta^n} <0.
\]
Then $ \Tr(\beta^n)=-q$ 
for some rational number $ q>0 $. Moreover,
\[
\Nm(\beta^n)=	\beta^n(\beta')^n = r
\]
is also rational and positive. This gives us
\[
    \beta^{2n}+q\beta^n+r=0,
\]
and after clearing out the denominators we obtain a nontrivial partition of zero.
\end{proof}

\section{Range of the partition function}\label{s:3}
In this section, we prove Theorem~\ref{thm2}.
\begin{lemma}
\label{lemmaE}
Let $ \beta $ be a totally positive quadratic integer with minimal polynomial $ x^2-Ex+C $, $E,C\in\Z_{\geq 1}$, and let $n\geq 0$. If there exists a partition of $ E\beta^n $ different from $ E\beta^n $, then it is of the form
\[
    E\beta^n = \beta^{n+1}+a_n\beta^n+\dots+a_1\beta+a_0
\]
for some $ a_i \in \Z_{\geq 0} $.
\end{lemma}
\begin{proof}
We assume that $ x^2-Ex+C $ is the minimal polynomial of $ \beta $, in particular irreducible. Since the two roots $ \beta $ and $ \beta' $ are real, the discriminant $ \Delta = E^2-4C $ is positive and not a square. It follows that $ E \geq 3 $ because $ C $ is also positive. We have $ \Delta \equiv 0, 1 \pmod{4} $, and therefore $ \Delta \geq 5$.

For any $ \alpha \in \Z[\beta] $, there is a one-to-one correspondence between partitions of $ \alpha $ into powers of $ \beta $ and of $ \alpha' $ into powers of $ \beta' $. Thus we can assume that $ \beta $ is the larger of the two conjugates, and hence
\[
	\beta = \frac{E+\sqrt{\Delta}}{2} \geq \frac{3+\sqrt{5}}{2} > 2.
\]
Consider a partition
\[
    E\beta^n = a_j\beta^j+a_{j-1}\beta^{j-1}+\dots+a_1\beta+a_0,\qquad a_j \neq 0.
\]
We treat first the case $ j \geq n+2 $. We established that $ \beta > 2 $, so $ \beta^2 > 2\beta $. But
\[
	E\beta^n \geq a_j\beta^j  \geq \beta^{n+2}
\]
implies $ \beta+\beta' = E \geq \beta^2 > 2\beta $, hence $ \beta' > \beta $, a contradiction.

Secondly, we suppose $ j = n+1 $ and show that $ a_{n+1} = 1 $. Otherwise $ a_{n+1} \geq 2 $ and
\[
	E\beta^n \geq a_{n+1}\beta^{n+1} \geq 2\beta^{n+1}
\]
implies $ \beta+\beta' = E \geq 2\beta $, which is impossible.

Finally, if $ j \leq n $ and the partition $ a_j\beta^j+a_{j-1}\beta^{j-1}+\dots+a_1\beta+a_0 $ is different from $ E\beta^n $, the polynomial
\[
    -Ex^n+a_jx^{j}+a_{j-1}x^{j-1}+\dots+a_1x+a_0
\]
is nonzero and has at most one sign change in the sequence of its coefficients. By Descartes' rule of signs, it has at most one positive root. But we know that $ \beta $ and $ \beta' $ are roots, a contradiction.
\end{proof}

We make some preliminary remarks before we prove Theorem~\ref{thm2}. The totally positive quadratic integer $ \beta $ is a root of the polynomial $ x^2-Ex+C $, where $ E = \Tr\beta $ and $ C = \Nm \beta $ both lie in $\Z_{\geq 1}$. First, we observe that, for a totally positive quadratic integer $ \beta $, the condition $ \beta > 1 $ and $ \beta' > 1 $ is equivalent to $ C \geq E $:
\[
	\beta>1\text{ and }\beta'>1 \Leftrightarrow (\beta-1)(\beta'-1)>0 \Leftrightarrow \beta\beta'-(\beta+\beta')+1>0 \Leftrightarrow C+1 > E.
\]

Secondly, the condition $ \beta < 2 $ or $ \beta'<2 $ is equivalent to $ C < 2E-4 $. Without loss of generality, we can assume that $ \beta $ is the larger of the two conjugates, hence $ \beta > 2 $ as we observed in the proof of Lemma~\ref{lemmaE}. Under this assumption,
\[
	\beta'<2 \Leftrightarrow (2-\beta)(2-\beta')<0 \Leftrightarrow 4-2(\beta+\beta')+\beta\beta' < 0 \Leftrightarrow C < 2E-4.
\]

Thus, Theorem~\ref{thm2} can be equivalently stated as follows: If $ C \geq E $, then $ p_\beta(E\beta^n) \geq n+1 $, with equality if and only if $ C < 2E-4 $; or $ n \in \{0,1\} $; or $ n = 2 $ and $ C < 2E $.

\begin{proof}[Proof of Theorem~$\ref{thm2}$]
The totally positive integer $ \beta $ is a root of the polynomial $ x^2-Ex+C $, where $ E = \Tr\beta $ and $ C = \Nm\beta $. We assume $ \beta > 1 $ and $ \beta' > 1 $ throughout the proof.
By the remarks preceding the proof, we know that $ C \geq E $.

Suppose that we have a non-trivial partition of $ E\beta^n $. By Lemma~\ref{lemmaE}, the partition is necessarily of the form $ Q(\beta) $, where
\[
    Q(x) = x^{n+1}+c_nx^n+\dots+c_1x+c_0
\]
is a polynomial with non-negative integer coefficients. Since $ \beta $ is a root of $ Q(x)-Ex^n $, there exists some polynomial $ P(x) \in \Z[x] $ such that $ P(x)(x^2-Ex+C) = Q(x)-Ex^n $. We let
\[
    P(x) = a_{n-1}x^{n-1}+a_{n-2}x^{n-2}+\dots+a_1x+a_0
\]
and obtain the following (both necessary and sufficient) conditions on the coefficients of $ P(x)(x^2-Ex+C) $:
\begin{equation}\label{eq:system}
	\begin{aligned}
	a_{n-1}& = 1\\
	-Ea_{n-1}+a_{n-2}& \geq -E\\
	Ca_i-Ea_{i-1}+a_{i-2}& \geq 0,\qquad i = 2, \ldots, n-1\\
	Ca_1-Ea_0& \geq 0\\
    Ca_0& \geq 0.
  \end{aligned}
\end{equation}

Now let us show~\eqref{eqPartTrace}, i.e., that $ p_\beta(E\beta^n) \geq n+1 $ for every $ n \geq 0 $. 

If we choose $ j \in \{0, \dots, n-1\} $ and let $ a_i = 1 $ for $ i = j, \dots, n-1 $, $ a_i = 0 $ for $ i = 0, \dots, j-1 $, all the inequalities \eqref{eq:system} are satisfied. The choice of $ j $ gives at least $ n $ non-trivial partitions of $ E\beta^n $. This proves~\eqref{eqPartTrace}.

\medskip

Secondly, we assume $ \beta < 2 $ or $ \beta' < 2 $ and show $ p_\beta(E\beta^n) = n+1 $ for every $ n \geq 0 $. Without loss of generality, we can suppose that $ \beta $ is the larger of the two conjugates, hence $ \beta' < 2 < \beta $. By the remarks preceding the proof, we know that $ C < 2E-4 $.

We claim that if $ i = 2, \ldots, n-1 $, then
\begin{equation}
\label{eqAi}
    a_{i-1}\geq 0\text{ and }a_{i-1} \geq 2a_i\quad\Longrightarrow\quad a_{i-2} \geq 2a_{i-1}.
\end{equation}
Indeed,
\[
    a_{i-2} \geq Ea_{i-1}-Ca_i \geq \left(E-\frac{C}{2}\right)a_{i-1} \geq 2a_{i-1},
\]
where the last inequality is equivalent to $ 2^2-E\cdot 2+C \leq 0 $, which is satisfied because $ \beta' < 2 < \beta $. This proves~\eqref{eqAi}.

Now suppose that $ a_i \leq 1 $ and $ a_{i-1} \geq 2 $ for some $ i = 1, \ldots, n-1 $. Then $ a_{i-1} \geq 2a_i $ and by a repeated application of~\eqref{eqAi}, we obtain $ a_0 \geq 2a_1 $. Since $ C < 2E $, it follows that
\[
    Ca_1-Ea_0 \leq (C-2E)a_1 < 0,
\]
a contradiction. In view of the fact that $ a_{n-1} = 1 $, we proved $ a_i \leq 1 $ for all $ i $.

Next, assume that $ a_i \leq 0 $ but $ a_{i-1}>0 $ for some $ i = 2, \ldots, n-1 $. Then
\[
    a_{i-2} \geq -Ca_i+Ea_{i-1} \geq Ea_{i-1} \geq 2a_{i-1},
\]
and we arrive at the same contradiction as before. Therefore, if $ a_i < 0 $ for some $ i $, then $ a_j \leq 0 $ for all $ j = 0, \ldots, i $. Let $ j $ be the smallest index for which $ a_j < 0 $. The coefficient of $ x^j $ in the polynomial $ P(x)(x^2-Ex+C) $ is
\[
    Ca_j-Ea_{j-1}+a_{j-2} = Ca_j < 0,
\]
a contradiction. This proves that all the coefficients $ a_i $ are non-negative. Additionally, if $ a_i = 0 $ for some $ i $, then $ a_j = 0 $ for every $ j = 0, \ldots, i $. The only remaining possibility is that there exists some index $ j $ in the set $ \{0, \ldots, n-1\} $ such that $ a_i = 1 $ for $ i = j, \ldots, n-1 $ and $ a_i = 0 $ for $ i = 0, \ldots, j-1 $. Hence the only partitions of $ E\beta^n $ are the $ n+1 $ partitions that we found before.

\medskip

If $ n =0,1 $, then Lemma \ref{lemmaE} quickly implies that $ p_\beta(E\beta^n) = 1,2 $ (respectively).

If $ n = 2 $, then the conditions \eqref{eq:system} on the coefficients $ a_i $ become
\begin{align*}
	a_1& = 1\\
	-Ea_1+a_0& \geq -E\\
	Ca_1-Ea_0& \geq 0\\
    Ca_0& \geq 0.
\end{align*}
If $ C < 2E $, the only solutions $ (a_1, a_0) $ are $ (1, 0) $ and $ (1, 1) $. On the other hand, if $ C \geq 2E $, there is an additional solution $ (a_1, a_0) = (1, 2) $.

\medskip

Finally, we prove that $ p_\beta(E\beta^n) > n+1 $ if $ n \geq 3 $ and the condition $ \beta<2 $ or $ \beta'<2 $ is not satisfied. In this case, we have $ C \geq 2E-4 $. If $ C = 2E-4 $, then the discriminant equals $ \Delta = E^2-4C = E^2-4(2E-4) = (E-4)^2 $, hence the polynomial is not irreducible. It remains to consider the case $ C \geq 2E-3 $. The condition $ \Delta = E^2-4C>0 $ implies $ E \geq 3 $. Using $ C \geq 2E-3 $, we get also $ E^2-8E+12 > 0 $, hence $ E > 6 $.

We show that the coefficients $ a_i $ can also take the values $ a_i = 1 $ for $ i = 2, \dots, n-1 $, $ a_1 = 2 $, and $ a_0 = 3 $. If $ i = 3, \dots, n-1 $, then
\[
	Ca_i-Ea_{i-1}+a_{i-2} \geq C-E+1 \geq 1,
\]
if $ i = 2 $, then
\[
	Ca_2-Ea_1+a_0 = C-2E+3 \geq 0
\]
by assumption, and
\[
	Ca_1-Ea_0 = 2C-3E \geq 2(2E-3)-3E = E-6 > 0.
\]
Thus inequalities \eqref{eq:system} are satisfied and  $ p_\beta(E\beta^n) \geq n+2 $.
\end{proof}

\begin{proof}[Proof of Corollary~$\ref{cor3}$]
In order to use Theorem~$\ref{thm2}$ (as restated just before its proof above), 
we will show there exist infinitely many elements $ \beta = a+b\sqrt{D} $ in $ \Z[\sqrt{D}] $ such that
\[
	2a \leq a^2-Db^2 < 4a-4.
\]

This is equivalent to
\[
	1 \leq (a-1)^2-Db^2 < 2a-3,
\]
and 
there are indeed infinitely many pairs $ (a, b) $ of positive integers satisfying these inequalities because Pell's equation $ x^2-Dy^2 = 1 $ has infinitely many solutions.
\end{proof}

Now we focus on the assumptions $ \beta > 1 $ and $ \beta' > 1 $ in Theorem~\ref{thm2} and show what happens when they are not satisfied.
Clearly, when $\beta$ a totally positive (quadratic) integer, then at least one of its conjugates is greater than 1, and so we can without loss of generality assume that $ \beta > 1 > \beta' $.

\begin{prop}
\label{propOpt1}
Let $ \beta $ be a totally positive quadratic integer and suppose that $ \beta > 1 > \beta' $. Then there is $ c_\beta >0 $ such that $ p_\beta((\Tr\beta)\beta^n)) \leq c_\beta $ for all $ n \in \Z_{\geq 0} $.
\end{prop}
\begin{proof}
Consider the reciprocal $ \gamma = \frac{1}{\beta} $. The conjugate of $ \gamma $ satisfies $ \gamma' = \frac{1}{\beta'} > 1 $. By Theorem~\ref{thm5}, $ p_{\frac{1}{\beta}}(\alpha) < \infty $ for every $ \alpha \in \R $. However, by Lemma~\ref{lemmaE}, every partition of $ (\Tr\beta)\beta^n $ into non-negative powers of $ \beta $ except for $ (\Tr\beta)\beta^n $ itself is of the form
\[
	(\Tr\beta)\beta^n = \beta^{n+1}+a_n\beta^n+\dots+a_1\beta+a_0.
\]
If we divide by $ \beta^{n+1} $, we find
\[
	\frac{\Tr\beta}{\beta} = 1+a_n\frac{1}{\beta}+\dots+a_1\frac{1}{\beta^n}+a_0\frac{1}{\beta^{n+1}}.
\]
Thus $ p_\beta((\Tr\beta)\beta^n) \leq p_{\frac{1}{\beta}}(\frac{\Tr\beta}{\beta}) $.
\end{proof}

\section{Questions}\label{s:4}

We conclude with several open questions.
\begin{enumerate}
\item Is it possible to generalize Theorem~\ref{thm1} to algebraic numbers of higher degrees? 
I.e., does the converse to Proposition \ref{propFin} hold? 

While this paper was under review, Dubickas~\cite{Du2} resolved this question in the affirmative (even more generally, for arbitrary complex $\beta$).

\item Let $ \beta $ be a totally positive quadratic integer. Is it true that for every $ n \in \Z_{\geq 0} $, there exists $ \alpha \in \R $ such that $ p_\beta(\alpha) = n $? We proved this under additional assumptions in Theorem~\ref{thm2}. It would be interesting to know the answer at least for specific examples, say $ \beta = 2+\sqrt{2} $ (where it seems that the answer is perhaps ``yes'').
\item Does there exist a quadratic integer $ \beta $ which is not totally positive but the range of the partition function $ p_\beta $ is all of $ \Z_{\geq 0} $? And what happens in the case of algebraic numbers $ \beta $  of higher degrees?
\item Does $p_\beta$ satisfy any interesting congruence properties (at least concerning its parity)? Proposition \ref{propRec} may provide a possible starting point.
\end{enumerate}

\section*{Declarations}

\textbf{Data sharing:} Not applicable to this article as no datasets were generated or analysed during the current study.

\noindent \textbf{Competing interests:} The authors have no competing interests to declare that are relevant to the content of this article.

\noindent \textbf{Funding:} The authors were supported by Czech Science Foundation (GA\v{C}R) grant 21-00420M, Charles University Research Centre program UNCE/SCI/022, and GA UK No. 742120.

\end{document}